\documentclass[11pt]{article}   
\usepackage[english]{babel}

\usepackage{amssymb,amscd, latexsym}   

\usepackage{mathrsfs}
\usepackage[all]{xy}

\usepackage{amsfonts}

\usepackage{euscript}
\usepackage{amsmath}
\usepackage{ifpdf}
\usepackage{amsthm}

\LARGE\textwidth=6in
\newcommand{\lar}{\longrightarrow}
\newcommand{\surjects}{\twoheadrightarrow}

\newtheorem{Theorem}{Theorem}[section]

\newtheorem{Proposition}[Theorem]{Proposition}
\newtheorem{Remark}[Theorem]{Remark}
\newtheorem{Example}[Theorem]{Example}
\newtheorem{Conjecture}[Theorem]{Conjecture}
\newtheorem{Definition}[Theorem]{Definition}
\newtheorem{Question}[Theorem]{Question}
\def\sqr#1#2{{\vcenter{\hrule height.#2pt
        \hbox{\vrule width.#2pt height#1pt \kern#1pt
            \vrule width.#2pt}
        \hrule height.#2pt}}}
\def\phi{\varphi}

\def\sqr#1#2{{\vcenter{\hrule height.#2pt
        \hbox{\vrule width.#2pt height#1pt \kern#1pt
            \vrule width.#2pt}
        \hrule height.#2pt}}}

\def\phi{\varphi}
\def\Rees{{\cal R}}
\def\hht{{\rm ht}\,}

\def\fm{{\mathfrak m}}

\def\pp{{{\mathbb P}}}

\def\cl#1{{\cal #1}}
\def\rk{\rm rank}

\begin{document}
\begin{center}
{\Large{\bf The Aluffi algebra of the Jacobian of points in projective space: torsion-freeness}} \\

\vspace{0.3in}

{\large\sc Abbas  Nasrollah Nejad\footnote{This research was in part supported by a grant from IPM (No.900130067).}}\quad
Aron  Simis\footnote{Supported by a grant from CNPq (301061/2009-2)
and a CAPES Senior Visiting Fellowship.} \quad Rashid Zaare-Nahandi

\end{center}


\begin{abstract}
The algebra in the title has been introduced by P. Aluffi. Let $J\subset I$ be ideals in the commutative ring $R$. The (embedded) Aluffi algebra of $I$ on $R/J$  is an intermediate graded algebra
between the symmetric algebra and Rees Algebra of the ideal $I/J$ over $R/J$. A pair of ideals has been dubbed
an Aluffi torsion-free pair if the surjective map of the Aluffi algebra of $I/J$ onto the Rees algebra of $I/J$ is injective. In this paper we focus on the situation where $J$ is the ideal of points in general linear position in projective space and $I$ is its Jacobian ideal.
\end{abstract}

\section{Introduction}
In \cite{aluffi} Aluffi introduced a graded algebra for the purpose of defining a
characteristic cycle of a hypersurface in parallel to the well-known conormal cycle in intersection theory.
Inspired by this construction, in \cite{AA} the first two authors have explored its algebraic side,
naming it  the {\em Aluffi algebra} of a pair of ideals $J\subset I$ (or of $I$ on $R/J$).
A little later, the torsion-freeness problem stated in \cite{AA} has been considered by the first and the third authors in  \cite{AR} for a special class of ideals.
However, by and large this question is widely open.

Let us expand a little on this question.
By definition, the (embedded) Aluffi algebra is
$${\cal A}_{_{R/J}}(I/J):={\cal S}_{R/J}(I/J)
\otimes_{{\cal S}_R(I)}\Rees_R(I)\simeq \bigoplus_{t\geq 0} I^t/JI^{t-1},$$
where ${\mathcal S}_B({\mathfrak a})$ and ${\mathcal R}_B({\mathfrak a})$  denote respectively the symmetric and the Rees algebra of the ideal ${\mathfrak a}$ in the ring $B$.
Clearly, there are natural surjections
 ${\cal S}_{R/J}(I/J)\surjects {\cal A}_{_{R/J}}(I/J)\surjects\Rees_{R/J}(I/J)$.
 The kernel of the rightmost surjection, called the
{\em module of Valabrega--Valla} has appeared before in a different context (see
\cite{VaVa}, also \cite[5.1]{Wolmbook1}):
\begin{equation}\label{vava}
 {\cl V}\kern-5pt {\cl V}_{J\subset I}=\bigoplus_{t\geq 2} \frac{J\cap
I^t}{JI^{t-1}}.
\end{equation}
As it turns out, provided $I$ has a regular element module $J$, the Valabrega--Valla module is the torsion of the Aluffi algebra (\cite[Proposition 2.5]{AA}) and, consequently, the Rees algebra of $I/J$ is the Aluffi algebra modulo its torsion.
We say that the pair of ideals $J\subset I$ is
{\em {\rm (}Aluffi{\rm )} torsion-free} if ${\cl V}\kern-5pt {\cl V}_{J\subset I}=\{0\}$.
Dealing directly with the  Valabrega--Valla module makes the structure of the Aluffi algebra itself slightly invisible. On the bright side, the results get simplified since for an ideal $I$ of quadrics as considered in this paper the heavy work is transferred to the nature of the Jacobian ideal of $I$.
Besides, the existence of non-trivial torsion is often delivered at the level of degree $2$ of ${\cl V}\kern-5pt {\cl V}_{J\subset I}$.

In this work we focus on the case where $J\subset R=k[x_0,\ldots,x_n]$ denotes the ideal of a set of points
in projective space $\pp^n=\pp^n_k$ over an algebraically closed field $k$, and $I$ denotes the Jacobian ideal of $J$,
i.e., $I=(J,I_n(\Theta))$ where $I_n(\Theta)$ is the ideal of $n$-minors of the Jacobian matrix $\Theta$ of a minimal set of homogeneous generators of $J$. We will  restrict ourselves to the case where the number of points does not exceed $2n$, in which
case the ideal is generated by forms of degree $\leq 2$ (\cite[Theorem1.4]{harris}), consequently $\Theta$ has linear entries.
In addition, the standing assumption will be that the points are in general linear position.

Now, quite generally, suppose that $J$ is an ideal generated by $2$-forms in the standard graded polynomial ring $R=k[x_0,\ldots,x_n]$. In this situation, the Jacobian matrix $\Theta$ has linear entries throughout. If $\hht(J)\geq 2$ then, as a particular case of \cite[Example 2.19]{AA} (see also \cite[Proposition 1.5]{AR}),
knowing that the ideal $I_r(\Theta)$ is the $r$th power of the irrelevant maximal ideal of $R$,
implies that the pair $J\subset (J,I_r(\Theta))$
is Aluffi torsion-free.

The overall tactics we employ follow this path.
However, there are some cases where $I_r(\Theta)$ is smaller. These exceptions require a special treatment since the pair $J\subset I$ may still be torsion-free.
The finer analysis crosses recent examples of Gorenstein ideals and in one case the underlying geometry has a classical flavor interwoven with additional results from commutative algebra.

The basic preliminary statement of the paper is Theorem~\ref{IGP}, while the main results concerning the central matter are Proposition~\ref{n+2points}, Proposition~\ref{n+3points} and Theorem~\ref{hyperplane_linear}.

\section{Ideal of points generated by quadrics}

Let $R=k[\mathbf{X}]=k[x_0,\ldots,x_n]$ denote a standard graded polynomial ring over a field $k$, let $J\subset R$ be a homogeneous ideal and let $I\subset R$ stand for the Jacobian ideal of $J$, by which we always mean the ideal $(J,I_r(\Theta))$ where $r=\hht(J)$ and $I_r(\Theta)$ stands for the {\em critical ideal} of $J$, i.e., the ideal generated by the $r$-minors of the Jacobian matrix $\Theta$ of a  set of generators of $J$. (It is well-known that the ideal  $(J,I_r(\Theta))/J\subset R/J$ does not depend on the choice of generators of $J$.)

We will henceforth focus on the case of an ideal of points in projective space  generated by $2$-forms.
Let $\Gamma=\{p_1,\ldots,p_s\}$ be a set of distinct points of $\pp^n=\mathbb{P}_k^n$, where $k$
is an algebraically closed field and $n\geq 2$. The defining ideal of $\Gamma$ is the ideal $J=\cap_{i=1}^s I(p_i)$
where $I(p_i)$ is the prime ideal of $p_i$ -- since we are assuming that $k$ is algebraically closed then $I(p_i)$ is generated by $n$ linear forms.
Note that $R/J$ is a reduced ring of dimension one, hence is a Cohen-Macaulay ring. We say
that the points in $\Gamma$ are in general linear position if either $s\leq n$ and the points span a $\mathbb{P}^{s-1}$, or else  $s\geq n+1$, in which case no subset of $n+1$ points of $\Gamma$ is contained in a hyperplane of $\mathbb{P}^n$.
We will often use the following facts without further ado: (1) the Aluffi torsion-freeness is invariant under a projective change of coordinates; given two sets each consisting of the same number $s\leq n+2$ of points in general linear position, then there is a projective change of coordinates carrying one onto the other.

\subsection{Results for arbitrary $n$}

The following preliminary result will allow us to focus on the case where the number $s$ of points
is at least $n+2$.

\begin{Proposition}\label{less_than}
Let $J\subset R$ denote the ideal of $1\leq s\leq n+1$ points in general linear position in $\pp^n$. The pair $J\subset I=(J,I_n(\Theta))$ is torsion-free if and only if $s\neq 2$.
\end{Proposition}
\begin{proof}
The case $s=1$ is trivially torsion-free. For $s=2$, by a projective change of coordinates,
the ideal $J$ will be $J=(x_0x_1,x_2,x_3,\ldots,x_n)$
with Jacobian ideal $I=(J,I_n(\Theta))=(x_0,\ldots,x_n)$.
Then, for example, $x_0x_1\in J\cap I^2\setminus JI$.

Let $3\leq s\leq n$. By a projective change of coordinates,
we may assume that the points are the coordinate points $[0:\cdots:0:1:0:\cdots:0]$, where 1 is in the $i$th position for $i=0,\ldots s-1$.
The defining ideal $J$ of these points is generated by square-free monomials of degree $2$ and $n-s+1$ variables as follows
$$
J=(x_ix_j : 0\leq i<j\leq s-1, \ \ x_s,\ldots,x_n).
$$
The Jacobian matrix $\Theta$ of $J$ is of the form
$$
\Theta=\left[
\begin{array}{c|ccc}
\Theta' & 0  & \cdots & 0 \\
\hline
0 &1  &\cdots & 0\\
\vdots  &&\\
0 &0  &\cdots &1
\end{array}
\right],
$$
where $\Theta'$ is the Jacobian matrix of the defining ideal of $s$ coordinate points in $\mathbb{P}^{s-1}$.
Then, $I_n(\Theta)=I_{s-1}(\Theta')= (x_0,\ldots,x_{s-1})^{s-1}$. The Jacobian ideal $I$ is
$$(J,I_n(\Theta))= (J,(x_0,\ldots,x_{s-1})^{s-1}) = (J, x_0^{s-1},\ldots,x_{s-1}^{s-1}).$$
 Set $\Delta=(x_0^{s-1},\ldots,x_{s-1}^{s-1})$.

 \smallskip

 {\sc Claim:}
 $J\cap \Delta^t\subseteq JI^{t-1}$.

  \smallskip

The proof uses the algorithmic  procedure for intersection of monomial ideals.
Namely, setting $J_1=(x_s,\ldots,x_n)$ and $J_2=(x_ix_j\ \ 0\leq i<j\leq s-1)$, one has
 $$
 J\cap \Delta^t = J_1 \cap \Delta^t +J_2 \cap \Delta^t\quad , \quad J_1\cap \Delta^t=J_1\Delta^t\subset J\Delta^{t-1}
 $$
and
 \begin{eqnarray}
 \nonumber J_2\cap \Delta^t&=&J_2\cap (x_0^{t(s-1)},\ldots,x_{s-1}^{t(s-1)})+J_2\cap (x_0^{\alpha_0(s-1)}\cdots x_{s-1}^{\alpha_{s-1}(s-1)}\ \ | \sum_{i=0}^{s-1}\alpha_i=t)\\
 \nonumber &=&\sum_{k=0}^{s-1}J_2\cap (x_k^{t(s-1)})+(x_0^{\alpha_0(s-1)}\cdots x_{s-1}^{\alpha_{s-1}(s-1)}\ \ | \sum_{i=0}^{s-1}\alpha_i=t).
 \end{eqnarray}
 For the first summand it obtains
 \begin{eqnarray}
 \nonumber (x_ix_j)\cap (x_k^{t(s-1)})=
  \begin{cases}
  x_k^{t(s-1)}x_j=(x_kx_j)x_k^{t(s-1)-1} & {\rm if} \;i=k\\
   x_ix_k^{t(s-1)}=(x_ix_k)x_k^{t(s-1)-1} &  {\rm if} \;j=k\\
  (x_ix_j)x_k^{t(s-1)} &  {\rm if} \;k\neq i,j,
  \end{cases}
\end{eqnarray}
hence $\sum_{k=0}^{s-1}J_2\cap (x_k^{t(s-1)})\in JI^{t-1}$.
Since $s\geq 3$, the second summand belongs to $JI^{t-1}$ and the claim is proved.

Assume next that $s=n+1$. By a projective change of coordinates, we may
assume that the given points are the coordinate points in $\mathbb{P}^n$.
The defining ideal is generated
by all degree $2$ square-free monomials $x_ix_j$, $0\leq i< j\leq n$.
Since this ideal is the edge ideal of a complete graph, the assertion
follows as in \cite[Example 3.4(i)]{AR}.
\end{proof}

In the above proposition, if the assumption that the points are in general linear position is omitted, the assertion may fail, as shown in the following example.

\begin{Example}\rm
The points $(0,1,0), (0,0,1), (0,1,1) \in \pp^2$ lie on the straight line $\{x_0=0\}$ and the ideal of these points is $J=(x_0,x_1x_2(x_1-x_2))$. A calculation with \cite{Macaulay1} shows that $J\cap I^2\not\subset JI$ -- e.g., $x_1x_2^3(x_1-x_2)\in J\cap I^2\setminus JI$.
\end{Example}

\subsection{Explicit generators}

As remarked earlier, when $3\leq s\leq 2n$ points are in general linear position in $\pp^n$ then the corresponding ideal of points $J$ is generated by quadrics.
In the next proposition we add further precision to this fact.

\begin{Theorem}\label{IGP}
Let $\Gamma$ be a set of $n+2\leq s\leq 2n$ points in general linear position in $\mathbb{P}^n$.
Let $t$ be an integer running in the interval $[2n-s+1,n-1]$.
Then the corresponding ideal $J$ of points  is minimally generated by the quadrics of the form
$$
g_{ij}=x_ix_j+\sum_{t=2n-s+1}^{n-1}\alpha^{(t)}_{ij} x_{t}x_n, \ \ (i,j)\in\Lambda, $$
where $$\Lambda:=\{(i,j)\in \mathbb{N}\times  \mathbb{N}\,|\,(i,j)\neq (t,n)\;\forall t\}.
$$
and $\alpha^{(t)}_{ij}\in k$ are uniquely determined by the coordinates of the  points in $\Gamma$.
In particular, the minimal number of generators $\mu(J)$ of $J$ is $|\Lambda|={{n+1}\choose {2}}-(s-(n+2))={{n+2}\choose {2}}-s$, hence lies in the interval $ {{n}\choose {2}}+1\leq \mu(J)\leq {{n+1}\choose {2}}+1.$
\end{Theorem}
\begin{proof}
Since the points in $\Gamma$ are in general linear position,  we may assume that $n+1$ of them are the coordinate points and an $(n+2)$nd point is $[1:1:\ldots:1]$.
For the remaining points, write $[a_{h,0}:a_{h,1}:\ldots:a_{h,n-1}:1]$ for $n+3\leq h\leq s$.
Consider the system of equations in the unknowns $\alpha^{(t)}_{ij}\in k \,(0\leq i<j \leq n)$:
\begin{eqnarray}\label{system}
 \begin{cases}
 1+\sum_{t=2n-s+1}^{n-1}\alpha^{(t)}_{ij}  =0&  \\
 a_{h,i}a_{h,j}+\sum_{t=2n-s+1}^{n-1} \alpha^{(t)}_{ij}  a_{h,t}=0,  &  (n+3\leq h\leq s).
\end{cases}
\end{eqnarray}

To find $\alpha_{ij}^{(t)}$ it is enough to solve the following matrix equations
\begin{equation}\label{syst}
\left[
\begin{array}{cccc}
1 & 1 & \cdots & 1\\
a_{n+3,t} & a_{n+3,t+1} & \cdots & a_{n+3,n-1}\\
\vdots &\vdots & \cdots& \vdots\\
 a_{s,t}& a_{s,t+1} & \cdots & a_{s,n-1}
\end{array}
\right]
\left[\begin{array}{c}
\alpha_{ij}^{(t)}\\
\alpha_{ij}^{(t+1)}\\
\vdots\\
\alpha_{ij}^{(n-1)}
\end{array}\right]=
\left[\begin{array}{c}
-1\\
-a_{n+3,i}a_{n+3,j}\\
\vdots\\
-a_{s,i}a_{s,j}
\end{array}\right].
\end{equation}

Now, since the points are in general linear position, any $(n+1)$-minor of the following matrix is nonzero.
$$
M=\left[
  \begin{array}{ccccccccc}
     1 & 0 & \cdots& 0 & 0      & 1      & a_{n+3,0} & \cdots & a_{s,0} \\
     0 & 1 & \cdots& 0 & 0      & 1      & a_{n+3,1} & \cdots & a_{s,1} \\
\vdots &   & \ddots&  & \vdots & \vdots & \vdots    &        & \vdots  \\
     0 & 0 & \cdots& 1 & 0      & 1      & a_{n+3,n-1} & \cdots & a_{s,n-1}\\
     0 & 0 & \cdots& 0 & 1      & 1      & 1           & \cdots & 1
  \end{array}
\right] .
$$
Then the minor
$$
\left[
\begin{array}{ccccccccc}
1 & 0 & \cdots & 0 & 0 & 0& \cdots & 0 &0\\
0 & 1 & \cdots & 0 & 0 & 0 & \cdots & 0 & 0\\
\vdots & & \ddots & \vdots & & \vdots  & & \vdots & \vdots\\
0 & 0 & \cdots & 1 & 0 & 0 & \cdots &0 & 0\\
0 & 0 & \cdots & 0 & 0 & 0 & \cdots & 0 & 1\\
a_{n+3,0}&a_{n+3,1}&\cdots& a_{n+3,t-1}& a_{n+3,t} & a_{n+3,t+1} & \cdots & a_{n+3,n-1}&1\\
\vdots &&&\vdots&\vdots &\vdots & \cdots& &\vdots\\
a_{s,0}&a_{s,1}&\cdots&a_{s,t-1}& a_{s,t}& a_{s,t+1} & \cdots & a_{s,n-1}&1
\end{array}
\right]
$$
is nonzero, which implies that the determinant of the first matrix in (\ref{syst}) does not vanish. Therefore, the system (\ref{syst}) has unique solution. Furthermore, by Cramer's rule, $\alpha_{in}^{(t)}\neq 0$ for $0\leq i\leq 2n-s$ and $2n-s+1 \leq t\leq n-1$.

Consider the ideal $J'\subset R$ generated by the quadrics $g_{ij}$ as in the statement, where the coefficients $\alpha^{(t)}_{ij}\in k \,(0\leq i<j \leq n)$ are the uniquely determined solutions of (\ref{system}).
Clearly, the generators of $J'$ vanish on $\Gamma$ and $J$ is a radical ideal. Therefore, $J'\subseteq J$.
We show that $J$ and $J'$ have the  same Hilbert function, hence must be equal.
Now, one knows by \cite{GO} that the ideal $J$ of a set of points in general linear position in $\mathbb{P}^n$ has maximal Hilbert function,  that is
$$
\dim_k (R/J)_t=s, \ \ t>0.
$$


\smallskip

As for $J'$, we claim that its Gr\"obner basis  with respect to the deg-revlex term ordering  with $x_0>x_1>\ldots> x_n$ is the set
\begin{eqnarray}\label{GbJ'}
G \cup \ \{x_{l}^2x_n+\sum_{t=2n-s+1, t\neq l}^{n-1} \beta^{(t)}_{l}x_{t}x_n^2, \ \ 2n-s+1\leq l\leq n-1\}.
\end{eqnarray}
where $G$ is the above generating set of $J'$ and each indexed $\beta^{(t)}_{l}$ is a certain polynomial expression of the $\alpha$'s.
For this, we consider the $S$-pairs of elements in this set.
First, we look at the $S$-polynomial of $g_{0l}$ and $g_{0n}$ for $2n-s+1\leq l\leq n-1$ is
$$
x_l\sum_{t=2n-s+1}^{n-1}\alpha^{(t)}_{0n} x_{t}x_n-x_n\sum_{t=2n-s+1}^{n-1} \alpha^{(t)}_{0n} x_{t}x_n
$$
which upon division by  the generators of $J'$ is reducible to $f_l=x_{l}^2x_n+\sum_{t=2n-s+1, t\neq l}^{n-1}\beta^{(t)}_{l}x_{t}x_n^2$, where $\beta^{(t)}_{l}$ is a certain polynomial like expression in the $\alpha$'s. Since the initial monomial of each of the $f_l$'s is not divisible by the initial term of any generator of $J'$ we add these polynomials to the generating set of $J'$.

Now consider the $S$-pairs $\{g_{ij},g_{kl}\}$ of the remaining generators of $J'$, where either $i\neq 0$ or $j\neq n$.
The initial monomial of any $g_{ij}$ in the generating set of $J'$  is $x_ix_j$.
Clearly, we may assume that
 $\{i,j\}\cap\{k,l\}\neq\varnothing$, as otherwise $x_ix_j$ and $x_kx_l$ are relatively prime.
Say, $i=k$ and $j<l$. In this case, the $S$-polynomial of $g_{ij}$ and $g_{kl}$ is $x_j(\sum_{t=2n-s+1}^{n-1}\alpha^{(t)}_{ij} x_{t}x_n)-x_l(\sum_{t=2n-s+1}^{n-1}\alpha^{(t)}_{kl} x_{t}x_n)$. The monomial $x_jx_tx_n$ is divisible by initial term of $g_{jt}$ or $f_j$ and the monomial $x_lx_tx_n$ is divisible by initial term of $g_{lt}$ or $f_l$. Thus, the remainder  of this polynomial upon division by the augmented generating set of $J'$ is zero.

The argument for the cases where $i=l$, $j=k$ or $j=l$ is entirely similar. As for a pair $\{g_{ij}, f_l\}$, if $j\neq n$ then  their initial terms are relatively prime, hence assume $j=n$. In this case the $S$-polynomial is
$$
x_l^2\sum_{t=2n-s+1}^{n-1}\alpha_{ij}^{(t)}x_tx_n - x_i\sum_{t=2n-s+1}^{n-1}\beta_l^{(t)}x_tx_n^2. 
$$
Each term of this polynomial is divisible by the initial term of $g_{in}$ or of $f_l$. The $S$-polynomial of a pair $\{f_l, f_{l'}\}$ similarly reduces to zero.

Thus, the following set is a minimal generating set for the initial ideal of $J'$:
$$
\{x_ix_j,\ x_{t}^2x_n, \ 0\leq i<j\leq n, (i,j)\neq (t,n),\ 2n-s+1\leq t\leq n-1\}.
$$

Therefore for any $r>0$,
$$
\dim_k (R/J')_r=\#\{x_i^r, \ 0\leq i\leq n,\ x_{t}x_n^{r-1}, 2n-s+1\leq t\leq n-1\}=s,
$$
as stated.
In particular, $\mu(J)=\dim_k (J_2)=\dim_k R_2-\dim_k(R/J_2)={{n+2}\choose {2}}-s$ is the minimal number of generators of $J$.
\end{proof}

\section{Points in $\pp^2$}

\subsection{Exceptions}

We now consider the case $n=2$. For $s\leq 3$, the question is taken care by Proposition~\ref{less_than}.
\smallskip

The case of $s=n+2=4$ is surprisingly  more involved and it turns out that the pair is not Aluffi torsion-free.
As will be seen later on $n=2$ is the only dimension for which $n+2$ points in general linear position are such that  the Aluffi algebra of the pair $J\subset I$ is not torsion-free.

We can assume that the four points in general linear position
in the projective plane are the coordinate points
$(1:0:0),(0:1:0),(0:0:1)$ and the additional point $(1:1:1)$. By Theorem~\ref{IGP}, the defining ideal is $J=(xz-yz, xy-yz)$, while the Jacobian matrix of these $2$-forms is:
$$\Theta=\left(
           \begin{array}{ccc}
             z & -z & x-y \\
             y & x-z & -y \\
           \end{array}
         \right)
$$
Therefore, $I:=(J,I_2(\Theta))=(xy-xz, xz-yz,  xz+yz-z^2, -xy+y^2-yz, -x^2+xy+xz)$.

A computation with \cite{Macaulay1} gives that $J\cap I^2$ is minimally generated by $11$ quartics, while $JI$ is obviously generated by at most $10$ quartics.
Therefore, the pair $J\subset I$ is not torsion-free.

\smallskip

Strikingly enough, this example has a curious algebraic-geometric background.

In one end, the underlying algebra will tell us that the ideal $I$ belongs to the class of ideals  of $k[x,y,z]$ of finite colength minimally generated by $5$ quadrics which happen to be syzygetic in the sense of \cite[Section 2]{SV1}.
By \cite{abc}, these ideals are Gorenstein.

Let $\phi$ denote the $5\times 5$ skew-symmetric
matrix whose Pfaffians are the generators of $I$.
Pick a new set of indeterminates ${\bf T}=\{T_1,T_2,T_3,T_4,T_5\}$ (think of them as the homogeneous coordinates
of $\mathbb{P}^4$) and consider the entries of the matrix product ${\bf T}\cdot\phi$.
Next take the Jacobian matrix $\psi$ of these bihomogeneous polynomials of bidegree $(1,1)$ with respect
to $x,y,z$ -- the so-called {\em Jacobian dual matrix} of $\phi$ (\cite{Jadual}).
Note that this a $5\times 3$ matrix whose entries are linear forms in $k[{\bf T}]$.

By a known argument as in \cite{syl2}, one can show that the maximal minors of $\psi$ are polynomial
relations of the $5$ original quadrics.
Therefore, since $\dim k[I]=3$ the codimension of the ideal $I_3(\psi)$  is at most $2$.
It can further be shown that $I_3(\psi)$ is a prime ideal of codimension $2$.

But a lot more is true:

\begin{Proposition}
With the above notation we have:
 $${\mathcal R}_R(I)\simeq R[{\mathbf T}]/(I_1({\mathbf T}.\phi),I_3(\psi)).$$
Thus,  $I$ is an ideal of fiber type.
Moreover, ${\mathcal R}_R(I)$ has depth $1$ -- the lowest possible.
\end{Proposition}
\begin{proof}
Since $I_1({\bf T}\cdot\phi)$ defines the symmetric algebra of $I$ and as a consequence of
 the above discussion, the ideal
$(I_1({\bf T}\cdot\phi), I_3(\psi))\subset R[{\bf T}]$ is contained in a presentation ideal of $\Rees_R(I)$ on $R[{\bf T}]$
and, moreover, $I_3(\psi)$ is the homogenous defining ideal of $k[I]$.
We must show that the whole ideal $(I_1({\bf T}\cdot\phi), I_3(\psi))$ is a prime ideal of codimension $4$.

In the other end, consider the rational map $\mathbb{P}^2\dasharrow \mathbb{P}^4$ defined by generators of $I$ which is birational on to its image. By \cite[Theorem 2.4]{bir2003}  one has $\rk (\psi)\equiv 2 \pmod{I_3(\psi)}$ and moreover, the coordinates of any nonzero homogeneous syzygy of $\psi$ modulo $I_3(\psi)$ defines the inverse rational map. In particular, these forms are algebraically independent over $k$. Actually, they will generate an ideal of linear forms modulo $I_3(\psi)$. From this and from  \cite[Proposition 2.1]{bir2003} now follows that $(I_1({\bf T}\cdot\phi), I_3(\psi))$ is a presentation ideal of ${\mathcal R}_R(I)$ on $R[{\mathbf T}]$. For the proof that ${\mathcal R}_R(I)$ has depth $1$ see \cite[Theorem 2.1 (ii)]{abc}.
\end{proof}

\smallskip

It is possible to write donw the presentation ideal of the Aluffi algebra as well, based on the presentation ideal in the above proposition. Although hardly useful at this point, we can moreover compute the torsion of the Aluffi algebra, the latter being generated by two forms in degree $2$.

\medskip

To understand the underlying geometric content, consider the rational map $\mathcal{F}:\mathbb{P}^2\dasharrow \mathbb{P}^4$
defined by five sufficiently general quadrics $\mathbf{q}=\{q_1,q_2,q_3,q_4, q_5\}\subset R$.
It is classically
known that the image of this map is a surface obtained as a general projection of the
$2$-Veronese embedding of $\mathbb{P}^2$ in $\mathbb{P}^5$. Therefore, the integral closure
of the homogeneous coordinate ring
of the image (i.e., the $k$-subalgebra $k[\mathbf{q}]\subset R$
up to an obvious degree normalization) is the Veronese algebra $R^{(2)}$.
Write $P\subset k[{\bf T}]$ for the homogeneous defining ideal of the image of $\mathcal{F}$.
By geometric considerations, one knows that the homogeneous defining ideal of this
smooth surface is generated by $7$ cubic forms.
To subsume the geometry under the algebra, one checks that the ideal generated by $5$ sufficiently general quadrics is syzygetic. Perhaps remarkable is that then the cubic forms
can be taken to be the minimal generators of the ideal of maximal
minors of the well-structured $5\times 3$ matrix described above.

\subsection{When is the critical locus a power?}

Quite generally, suppose that $J\subset R$ is an ideal generated by $2$-forms. In this situation, the Jacobian matrix $\Theta$ has linear entries throughout. If $\hht(J)\geq 2$ then, as a particular case of \cite[Example 2.19]{AA} (see  \cite[Proposition 1.5]{AR}),
knowing that $I_r(\Theta)$ is the $r$th power of the irrelevant maximal ideal of $R$,
implies that the pair $J\subset (J,I_r(\Theta))$
is Aluffi torsion-free. It seemed reasonable to  conjecture in \cite{AR} that it is always the case that $I_r(\Theta)$ is $\fm$-primary if and only if it coincides with the $r$th power of $\fm=(x_0,\ldots,x_n)$ \cite[Conjecture 2.6]{AR}.
Unfortunately, this conjecture is not true in all its generality as shown in the following example.

\begin{Example}\label{2n_for3_goes_wrong} \rm
Consider the  $6$ points in $\pp_k^3$, which are written as columns of the following matrix:
$$
\left(
  \begin{array}{cccccc}
   1 & 0 & 0 & 0 & 1 & -1\\
0 & 1 & 0 & 0 & 1 & 2\\
0 & 0 & 1 & 0 & 1 & 3\\
0 & 0 & 0 & 1 & 1 & 1\\
  \end{array}
\right)
$$
\end{Example}
They are in general linear position since all the $4$-minors are nonzero.
By Theorem~\ref{IGP}, one has
$$J=(x_0x_1-5x_1x_3+4x_2x_3,x_0x_2-6x_1x_3+5x_2x_3,x_0x_3-4x_1x_3+3x_2x_3,x_1x_2-4x_1x_3+3x_2x_3).$$
A computation with \cite{Macaulay1} yields
\begin{enumerate}
\item[$\bullet$] $I=(J,x_1x_3^2, x_2x_3^2, x_0^3, x_1^3, x_2^3, x_3^3)$
\item[$\bullet$] $\fm ^3\subset I$
\item[$\bullet$] $\mu(I_3(\Theta))=16$ (hence $\fm^3\neq I_3(\Theta)$)
\end{enumerate}

Still, since $\fm ^3\subset I$ and the inclusion $I_3(\Theta)\subset \fm^3$ always holds, then $I=(J,\fm^3)$ -- hence, the pair $J\subset I$ is torsion-free by
\cite[Example 2.19]{AA}.

\begin{Remark}\rm
It may be contended that the above example, although in general linear position,  is not ``general enough".
However, a computation with random coordinates for the points yields the same result, so the failure is due to the nature of given data, no matter what sort of stronger general position notion is assumed.
\end{Remark}

Examining closely the data of the above example, the following might  be a more realistic question.

\begin{Question}\rm Let $J$ be generated in  degree $2$ and assume that $r:=\hht(J)>2$. If $I_r(\Theta)$ is $\fm$-primary
and $I:=(J,I_r(\Theta))$ contains at least the pure powers $x_0^r,\ldots,x_n^r$,  then $\fm^r\subset I$
{\rm (}and hence, $I=(J,\fm^r)${\rm )}.
\end{Question}

\section{Points in $\pp^n$}

\subsection{Points in $\pp^n\,$ ($n\geq 3$)}

In this part we assume that $n\geq 3$.
Recall the negative sort of result in Example~\ref{2n_for3_goes_wrong} for $s=n+3$, where it has been seen that the critical ideal $I_{n}(\Theta)$ is not always a power of the irrelevant $R_+=\fm$ even though the pair $J\subset (J, I_n(\Theta))$ is torsion-free.
Clearly, the inclusion $I_n(\Theta)\subset \fm^n$ always holds as they are both generated in degree $n$

Still, for $s=n+2$ one has:

\begin{Proposition}\label{n+2points}
Let $\Gamma$ be a set of $n+2$ distinct points in general position  in $\mathbb{P}^n$ with $n\geq 3$.
Let $J\subset R=k[x_0,\ldots,x_n]$ be the corresponding ideal of points.
Then $I_n(\Theta)=\fm^n${\rm ;} in particular the pair $J\subset (J,I_n(\Theta))$ is torsion-free.
\end{Proposition}
\begin{proof}
By Theorem~\ref{IGP}, up to a projective change of coordinates $J$ is generated by the following quadrics
\begin{equation}\label{n+2_points_equations}
\{x_ix_j-x_{n-1}x_n,\,  0\leq i<j\leq n,\, (i,j)\neq (n-1,n)\}.
\end{equation}
Then $J$ is the ``canonical''  submaximal ideal of quadrics of the square-free Veronese Cohen--Macaulay ideal $K=(x_ix_j, \,  0\leq i<j\leq n)$ (edge ideal of the complete $(n+1)$-graph).
The transposed Jacobian matrix $\Theta(K)^t$ of $K$ is the well-known Koszul matrix of $K$ ``without signs'', hence its ideal of $u$-minors is $\fm^u$ for every $u\leq n$.

After applying to $\Theta(K)^t$ elementary column operations consisting in subtracting the last column from the remaining columns, one easily sees that
$$
\Theta(K)^t=\left[
\begin{array}{c|c}
& 0\\
& 0\\
\Theta(J)^t & \vdots\\
& 0\\
& x_n\\
& x_{n-1}
\end{array}
\right].
$$
A straightforward calculation now gives the equality $I_n(\Theta(J))^t=I_n(\Theta(K))^t=\fm^n$.
\end{proof}

\begin{Remark}\rm
(1)
An alternative inductive argument would depend on writing
$$
\Theta(J)^t=\left[
\begin{array}{cccc|ccc}
x_1 & x_2 & \cdots & x_n & 0 & \cdots & 0 \\
\hline
&&&&&&\\
&&\ast & && \Theta'  &\\
&&&&&&
\end{array}
\right],
$$
where $\Theta'$ is the transposed Jacobian matrix of the ideal of suitable $n+1$ points in general linear position in
${\mathbb P}_k^{n-1}$ viewed in coordinate $x_1,\ldots,x_n$.
But the procedure would work as far down as from $n=4$ to $n=3$.
In  dimension $3$ a direct argument would be required, since the statement of the Proposition is false for $n=2$.

(2)
The ideal $J$ is Gorenstein, pretty much as in the full Veronese case, corresponding to the situation of finite colength  (\cite{abc}).
It would be interesting to study this class of $1$-dimensional Gorenstein ideals on itself.
\end{Remark}

We introduce a notion weaker than general linear position:

\begin{Definition}\rm
Let $\Gamma$ denote a set of $s\geq n+2$ distinct points in $\pp^n$.
We say that $\Gamma$ is in {\em hyperplane linear position} if it admits $s-1$ points spanning a hyperplane $H\subset \pp^n$ and in general linear position as points of the $\pp^{n-1}\simeq H$, while the remaining point of $\Gamma$ lies outside $H$.
\end{Definition}

\begin{Theorem}\label{hyperplane_linear}
Let $\Gamma$ be a set of $n+2$ distinct points in hyperplane linear position  in $\mathbb{P}^n$, where $n\geq 2$.
Let $J\subset R=k[x_0,\ldots,x_n]$ denote the corresponding ideal of points.
Then  the pair $J\subset (J,I_n(\Theta))$ is torsion-free.
\end{Theorem}
\begin{proof}
By a projective change of coordinates
one may assume that $H:\{x_n=0\}$. Further, by identifying $H$ with $\pp^{n-1}$, we may assume that the $n+1$ points lying on $H$ are the coordinate points of $H=\pp^{n-1}$ and the ``diagonal unit'' point,  stacking $0$ as the $n$th coordinate of each of these points as points of $\pp^n$.
In addition, concating the matrix of the coordinates of these $n+1$ points with the vector of coordinates of the remaining point and then applying elementary column operations, we may assume that the last point is $(0:0:\cdots:0:1)$.
This gives the matrix of points coordinates
$$
\left(
  \begin{array}{ccccc|c}
    1 & 0 & \cdots & 0 & 1 & 0 \\
    0 & 1 &\cdots  & 0 & 1& 0 \\
   \vdots & \vdots & \ddots & \vdots & \vdots & \vdots \\
    0 & 0 &\cdots  & 1& 1 & 0 \\
    \hline\\[-10pt]
    0 & 0 & \cdots & 0 & 0 & 1 \\
  \end{array}
\right)
.$$

The surprise in this weaker setup is that, even for $n=2$ the pair $J\subset I$ is torsion-free.
Indeed, an immediate calculation gives
$$J=(x_1,x_2)\cap (x_0,x_2)\cap (x_0-x_1,x_2)\cap (x_0,x_1)=(x_0x_2,x_1x_2,x_0x_1(x_0-x_1)).$$
(Note the exceptional behavior: there is a minimal generator of degree $3$.)

Another direct calculation yields  $I=(x_0^3,x_0^2x_1,x_0x_1^2,x_1^3,x_0x_2,x_1x_2,x_2^2)$.
Now, a computation with \cite{Macaulay1} yields that the relation type  of $I/J$ on $R/J$ is $2$ (note that this result is a bit elusive, as a set of minimal generators of $I$ contains one of $J$, yet the ideal of $R$ generated by the complementary subset of minimal generators of $I$ has relation type $3$). Then, by \cite[Corollary 2.17]{AA}
it suffices to  check that $J\cap I^2\subset JI$. An additional elementary computation with \cite{Macaulay1} yields this inclusion.

Now suppose that $n\geq 3$.
In this case we show that $I_n(\Theta(J))=\fm^n$ by a similar argument as in the proof of the preceding proposition.
For that, we need to know a set of minimal generators of $J$.

{\sc Claim:}   $J=(x_ix_j-x_{n-2}x_{n-1},x_ix_n, \ 0\leq i<j\leq n-1)$.

\smallskip

(Note that we cannot use Theorem~\ref{IGP} automatically since $\Gamma$ is not in general linear position, hence one needs a different approach.)

From the shape of the coordinates of the points, $J$ surely contains the ideal $J'$ generated by these $2$-forms.
In order to show that $J=J'$ we prove that the respective Hilbert functions coincide  (a debate one could avoid by providing a direct proof that  $J'$ is a saturated ideal of
multiplicity (degree) $n+2$).

\medskip

Write $\Gamma=\Gamma_1\cup \Gamma_2$ where $\Gamma_1$ is the set of $n+1$ points spanning $H$ and $\Gamma$ is the set consisting of the unique point not on $H$.
Since $\Gamma_1$ is in general linear position as points in $\pp^{n-1}$,
then Theorem~\ref{IGP} is applicable, hence as in (\ref{n+2_points_equations})
its ideal of points is
$$J_{\Gamma_1}=(x_ix_j-x_{n-2}x_{n-1},x_n,\ 0\leq i<j\leq n-1).$$
Clearly,  $J_{\Gamma_2}=(x_0,x_1,\ldots,x_{n-1})$. Since $J=J_{\Gamma_1}\cap J_{\Gamma_2}$,
one has a short exact sequence
$$
0\lar R/J\lar R/J_{\Gamma_1}\oplus R/J_{\Gamma_2}\lar R/(J_{\Gamma_1},J_{\Gamma_2})\lar 0.
$$
Direct inspection gives
$$
R/(J_{\Gamma_1},J_{\Gamma_2})\simeq k\ \ , \ \ R/J_{\Gamma_1}\simeq k[x_0,\ldots,x_{n-1}]/(x_ix_j-x_{n-2}x_{n-1}), \ \  R/J_{\Gamma_2}\simeq k[x_n].
$$
By the additive property of Hilbert function, we derive the Hilbert function of $R/J$:
$$
\mbox{Hilb}_{R/J}(0)=1,\  \mbox{Hilb}_{R/J}(1)=n+1,\  \mbox{Hilb}_{R/J}(t)=n+2,\ \   t\geq 2.
$$
For the Hilbert function of $R/J'$, we compute a Gr\"obner basis of $J'$ in the lex order with $x_0>x_1>\cdots> x_n$. By a similar argument as in proof on the Theorem~\ref{IGP}, we can show that a Gr\"obner basis is
 $$
 \{ x_ix_j-x_{n-2}x_{n-1},x_ix_n,\ x_{n-2}^2x_{n-1}-x_{n-2}x_{n-1}^2, \ \ \  0\leq i<j\leq n-1\}.
 $$
The initial ideal of $J'$ is generated by $\{x_ix_j, x_{n-2}^2x_n, 0\leq i<j\leq n, \ i\neq n-2\ \}$. The rest is as in
the end of the proof of Theorem~\ref{IGP}, thus showing that $J=J'$, as claimed.

\smallskip

We now argue that the pair $J\subset I$ is Aluffi torsion-free by showing that $I_n(\Theta(J))=\fm^n$, in pretty much the same way as was argued in the proof of Proposition~\ref{n+2points}.
Namely, we consider the squarefree Veronese ``hull'' of $J$:
$$K=(x_ix_j,\, 0\leq i< j\leq n)=(x_ix_j,\, 0\leq i< j\leq n-1; x_0x_n,\ldots, x_{n-1}x_n),$$
from which the above set of generators of $J$ is obtained by some obvious elementary transformations of the generators of $K$.

Accordingly, up to the same elementary operations applied to the corresponding columns of $\Theta(K)^t$, we get
$$
\Theta(K)^t=\left[
\begin{array}{c|c}
& 0\\
& 0\\
 & \vdots\\
\Theta(J)^t& 0\\
& x_{n-2}\\
& x_{n-1}\\
& 0
\end{array}
\right].
$$
A straightforward calculation now gives the equality $I_n(\Theta(J))^t=I_n(\Theta(K))^t=\fm^n$, where as before the last equality comes from the structure of $\Theta(K))^t$ as Koszul matrix without signs.
\end{proof}

\subsection{Points in $\pp^n\,$  ($n\geq 4$)}
\begin{Proposition}\label{n+3points}
Let $\Gamma$ be a set of $n+3$ distinct points in general linear position in ${\mathbb P}_k^{n}$ with
$n\geq 4$. Let $J\subset R=k[x_0,\ldots,x_n]$ be the defining ideal of $\Gamma$ and let
$I=(J,I_n(\Theta))$ stand for the Jacobian ideal of $J$. Then the pair $J\subset I$ is Aluffi torsion-free.
\end{Proposition}
\begin{proof}
Since $\Gamma$ is in general linear position,  we may assume that $n+1$ of them are the coordinate points and an $(n+2)$nd point is $[1:1:\ldots:1]$.
Let $[a_0:\ldots :a_{n-1}:1]$ denote the $(n+3)$rd point. Note that general linear position property of points implies that $a_{n-2}\neq a_{n-1}$.

By Theorem~\ref{IGP}, the defining ideal of these points
 is
$$
J=(x_ix_j+\alpha_{ij}^{(n-2)}x_{n-2}x_n-\alpha_{ij}^{(n-1)}x_{n-1}x_n\ , \ 0\leq i<j\leq n,\ (i,j)\neq (n-2,n),(n-1,n))
$$
where $\alpha_{ij}^{(n-2)}=\frac{a_{n-1}-a_ia_j}{a_{n-2}-a_{n-1}}$, $\alpha_{ij}^{(n-1)}=\frac{a_ia_j-a_{n-2}}{a_{n-2}-a_{n-1}}$.

\medskip

The transposed Jacobian matrix of $J$ has the form
$$
\Theta=\left[
\begin{array}{cccc|ccc}
x_1 & x_2 & \cdots & x_n & 0 & \cdots & 0 \\
\hline
&&&&&&\\
&&\ast & && \Theta'  &\\
&&&&&&
\end{array}
\right],
$$
where $ \Theta' $ is the transposed  Jacobian matrix of the defining ideal of $n+2$ points in general linear position
in ${\mathbb P}^{n-1}$ with coordinate $x_1,\ldots,x_n$.
Then the result follows by induction on $n$, except when $n=4$ and $s=4+3=7$ as in this case $\Theta'$ is the Jacobian matrix of $n+2=6$ points in $\pp^{n-1}$, when the minors of $\Theta'$ do not generate the entire power of $\fm$.
Thus, in this case one needs a direct argument, as follows.

According to Theorem~\ref{IGP}, a generating set of $J$ consists of the following polynomials.
\begin{gather*}
x_0x_1 + \alpha_{01}^{(2)}x_2x_4+\alpha_{01}^{(3)}x_3x_4 \quad , \quad x_1x_2 + \alpha_{12}^{(2)}x_2x_4+\alpha_{12}^{(3)}x_3x_4\\
x_0x_2 + \alpha_{02}^{(2)}x_2x_4+\alpha_{02}^{(3)}x_3x_4 \quad , \quad x_1x_3 + \alpha_{13}^{(2)}x_2x_4+\alpha_{13}^{(3)}x_3x_4\\
x_0x_3 + \alpha_{03}^{(2)}x_2x_4+\alpha_{03}^{(3)}x_3x_4 \quad , \quad x_1x_4 + \alpha_{14}^{(2)}x_2x_4+\alpha_{14}^{(3)}x_3x_4\\
x_0x_4 + \alpha_{04}^{(2)}x_2x_4+\alpha_{04}^{(3)}x_3x_4 \quad , \quad x_2x_3 + \alpha_{23}^{(2)}x_2x_4+\alpha_{23}^{(3)}x_3x_4.
\end{gather*}
The Jacobian matrix of $J$ therefore has the form:
$$
\Theta =
\left[
\begin{array}{c|cccc}
  x_1 & x_0 & \alpha_{01}^{(2)}x_4 & \alpha_{01}^{(3)}x_4 & \alpha_{01}^{(2)}x_2 + \alpha_{01}^{(3)}x_3 \\
  x_2 & 0 & x_0 + \alpha_{02}^{(2)}x_4 & \alpha_{02}^{(3)}x_4 & \alpha_{02}^{(2)}x_2 + \alpha_{02}^{(3)}x_3 \\
  x_3 & 0 & \alpha_{03}^{(2)}x_4 & x_0+\alpha_{03}^{(3)}x_4 & \alpha_{03}^{(2)}x_2 + \alpha_{03}^{(3)}x_3 \\
  x_4 & 0 & \alpha_{04}^{(2)}x_4 & \alpha_{04}^{(3)}x_4 & x_0+\alpha_{04}^{(2)}x_2 + \alpha_{04}^{(3)}x_3 \\
  \hline
  0 & x_2 & x_1+\alpha_{12}^{(2)}x_4 & \alpha_{12}^{(3)}x_4 & \alpha_{12}^{(2)}x_2 + \alpha_{12}^{(3)}x_3 \\
  0 & x_3 & \alpha_{13}^{(2)}x_4 & x_1+\alpha_{13}^{(3)}x_4 & \alpha_{13}^{(2)}x_2 + \alpha_{13}^{(3)}x_3 \\
  0 & x_4 & \alpha_{14}^{(2)}x_4 & \alpha_{14}^{(3)}x_4 & x_1+\alpha_{14}^{(2)}x_2 + \alpha_{14}^{(3)}x_3 \\
  0 & 0 & x_3+\alpha_{23}^{(2)}x_4 & x_2+\alpha_{23}^{(3)}x_4 & \alpha_{23}^{(2)}x_2 + \alpha_{23}^{(3)}x_3
\end{array}
\right],
$$
where the lower right block is the Jacobian matrix of six points in general linear position in ${\mathbb P}^3$ with coordinates $x_1,x_2,x_3,x_4$. By Example~\ref{2n_for3_goes_wrong}, $(x_1,x_2,x_3,x_4)^3\subseteq (J_{x_0},I_3(\Theta_{x_0}))$. Then, $(x_1,x_2,x_3,x_4)^4\subseteq (J,I_4(\Theta))$.  By changing the roles of $x_0$ and $x_1$, we get
$(x_0,x_2,x_3,x_4)^3\subseteq (J_{x_1},I_3(\Theta_{x_1}))$ and $(x_0,x_2,x_3,x_4)^4\subseteq (J,I_4(\Theta))$, where $J_{x_i}$ and $\Theta_{x_i}$ for $i=0,1$, denote the ideal and its Jacobian matrix of six points in coordinates $x_1,x_2,x_3,x_4$ and $x_0,x_2,x_3,x_4$, respectively.
By inspection one can  see that $x_0^3x_1, x_0^2x_1^2, x_0x_1^3\in (J,I_4(\Theta))$. Therefore, $(x_0,x_1,x_2,x_3,x_4)^4\subseteq (J,I_4(\Theta))$, thus yielding $I=(J,I_4(\Theta))=(J,\fm^4)$ and hence, $J\subseteq I$ is Aluffi torsion-free.
\end{proof}

We close with the following conjecture.
\begin{Conjecture}
Let $\Gamma$ be a set of $s=2n$ distinct points in general position in $\mathbb{P}^n$.
Let $J\subset R=k[x_0,\ldots,x_n]$ be the defining ideal of $\,\Gamma$ and let $I=(J,I_n(\Theta))$ stand for the
 Jacobian ideal of $J$. Then $I_n(\Theta)=\fm^n$. In particular, the pair $J\subset I$ is Aluffi torsion-free.
\end{Conjecture}
As a consequence of this conjecture, using a similar argument as the induction step in the proof of the Proposition~\ref{n+3points}, one can deduce the Aluffi torsion freeness of  $n+4\leq s\leq 2n$ points in general linear position in $\mathbb{P}^n$.

\noindent {\sc Abbas Nasrollah Nejad\\
Institute for Advanced Studies in Basic Sciences (IASBS),
P.O.Box 45195-1159, Zanjan 45137-66731, Iran\\
School of Mathematics, Institute for Research in Fundamental Sciences (IPM), P.O.Box 19395-5746, Tehran, Iran}\\
abbasnn@iasbs.ac.ir

\medskip

\noindent {\sc  Aron Simis\\
Departamento de Matem\'atica, Universidade Federal de Pernambuco, 50740-560 Recife, PE, Brazil, and Departamento de Matem\'atica, Universidade Federal da Paraiba, 58051-900 J. Pessoa, PB, Brazil}\\
aron@dmat.ufpe.br

\medskip

\noindent {\sc Rashid Zaare-Nahandi\\
Institute for Advanced Studies in Basic Sciences (IASBS),
P.O.Box 45195-1159, Zanjan 45137-66731, Iran}\\
rashidzn@iasbs.ac.ir

\end{document}